\documentclass[11pt]{article}
\usepackage[a4paper,margin=1in]{geometry}
\usepackage[shortlabels]{enumitem}
\setenumerate[1]{itemsep=0pt,partopsep=0pt,parsep=\parskip,topsep=0pt}
\usepackage[tworuled,linesnumbered,noline,noend]{algorithm2e}
\usepackage{amsmath,amsthm,tikz,float,amssymb,tikz,mathtools,bbm}
\usepackage[hidelinks]{hyperref}
\usepackage[numbers,sort&compress]{natbib}
\usepackage{algorithmic}

\allowdisplaybreaks

\newtheorem{theorem}{Theorem}[section]
\newtheorem{lemma}[theorem]{Lemma}

\newtheorem{corollary}[theorem]{Corollary}

\newcommand{\opt}{\mathrm{OPT}}

\newcommand{\be}{\mathbb{E}}
\newcommand{\bp}{\mathbb{P}}
\newcommand{\bo}{\mathbbm{1}}

\newcommand{\talg}{\mathrm{ALG}}

\newcommand{\diff}{\,\mathrm{d}}

\newcommand{\dx}{\,\mathrm{d}x}

\usepackage{authblk}
\begin{document}
\title{\textbf{Optimal Prophet Inequalities for Gain from Trade}}
\renewcommand\Authfont{\normalsize}
\renewcommand\Affilfont{\small}

\author[a,b]{Xujin Chen}
\author[a,b]{Xiaodong Hu}
\author[a]{Changjun Wang}
\author[c]{Qingjie Ye}

\affil[a]{SKLMS, Academy of Mathematics and Systems Science, Chinese Academy of Sciences, Beijing 100190, China}
\affil[b]{School of Mathematical Sciences, University of Chinese Academy of Sciences, Beijing 100049, China}
\affil[c]{School of Mathematical Sciences, Key Laboratory of MEA (Ministry of Education), Shanghai Key Laboratory of PMMP, Nantong Institute for Applied Mathematics and Artificial Intelligence, East China Normal University, Shanghai 200241, China}
\affil[ ]{\textit{\{xchen, xdhu, wcj\}@amss.ac.cn, qjye@math.ecnu.edu.cn}}
\date{}
\maketitle

\begin{abstract}
We initiate the study of \emph{prophet trading}, an online trading model in which a trader interacts with sellers and buyers who arrive in a uniformly random order and whose prices are drawn independently from a common known distribution. The trader aims to maximize the expected \emph{gain from trade} (GFT), with performance measured against an omniscient prophet that knows the realized arrival order and all prices in advance. Unlike the classical prophet inequality problem, the trader must make both purchasing and selling decisions while managing the inventory, which makes both the prophet benchmark and the analysis of online algorithms substantially more intricate.

For arbitrary numbers of buyers and sellers, we propose a simple threshold-based algorithm and prove a distribution-free competitive ratio of~2, which is the best possible. Our main technical contribution is an exact characterization of the inventory process induced by threshold trading. By expressing the algorithm's expected GFT in terms of the cumulative holding probability at buyer arrivals, we derive an exact formula for this inventory term, which serves as the foundation of our analysis and yields sharper guarantees in several important special cases.

For balanced trading (with $n$ sellers and $n$ buyers), we improve the competitive ratio to \((2n^2-n)/((n+4^{-n}-1)(n+1))\), which is strictly less than~2 for every \(n>1\). For the single-seller case, we sharpen the analysis of the fixed-threshold algorithm to obtain an asymptotic competitive ratio of \(2e^2/(e^2+1)\approx1.76\). Furthermore, by exploiting the additional structure of this setting, we design an adaptive-threshold algorithm whose competitive ratio approaches \(e/(e-1)\approx1.58\). Due to a symmetry property of our general algorithmic idea, the same \(1.76\)-competitive guarantee also holds for the single-buyer case.
\end{abstract}

\section{Introduction}

The prophet inequality problem \citep{krengel1977semiamarts,krengel1978semiamarts,samuelcahn1984comparison} is a classical optimal stopping problem. In its basic form, a decision-maker seeks to sell a single item to one of several buyers whose offers are drawn independently from known distributions and revealed sequentially. Upon observing each offer, the decision-maker must irrevocably decide whether to accept it. Accepting an offer terminates the process, whereas rejecting it permanently forfeits the opportunity to sell the item to that buyer. The benchmark is an omniscient prophet who knows all offers in advance, and the central objective is to design an online algorithm whose expected performance is within a constant factor of the prophet's.

\paragraph{Prophet trading model.}
One natural extension of the classical prophet inequality is to allow the decision-maker not only to sell items to buyers but also to acquire items from sellers. This leads to an online trading model in which purchasing and selling decisions are interleaved, and inventory management becomes an integral part of the decision process. Unlike the classical setting, purchasing an item creates future selling opportunities, whereas selling consumes inventory and may preclude future trades. Consequently, the decision-maker must balance immediate gains against future trading opportunities, a feature absent from the classical prophet inequality.

In this paper, we study the following \emph{prophet trading} model. The decision-maker, which we call the \emph{trader}, initially holds no item and trades with \(n\) buyers and \(m\) sellers. These buyers and sellers, collectively called \emph{agents}, arrive in a uniformly random order. Upon arrival, each agent reveals its valuation, which becomes the \emph{price} at which the trader may trade with that agent. After observing the price, the trader must immediately decide whether to trade with the current agent: it may purchase an item from a seller, or sell an item to a buyer provided that it currently holds one. We assume that the agents' prices are drawn independently from a common known distribution. The trader's objective is to maximize the expected \emph{gain from trade} (GFT), defined as the total price received from buyers minus the total price paid to sellers. Its performance is measured against an omniscient prophet, who knows the realized arrival order and all prices in advance and trades optimally on every realization.

\paragraph{Our contributions.}
The combination of online purchasing, online selling, and inventory constraints makes prophet trading qualitatively different from the classical prophet inequality and gives rise to a new class of online decision problems. We initiate the study of prophet trading under i.i.d.\ known price distributions and show that, despite these additional challenges, simple threshold-based algorithms achieve strong competitive guarantees.

 For the general setting with arbitrary numbers of buyers and sellers, we propose a simple threshold algorithm (Algorithm~\ref{alg:general}) and prove a distribution-free competitive ratio of \(2\), {which is the best possible in general.} A key technical ingredient is a new characterization of the inventory process under threshold trading, which expresses the trader's expected GFT through the cumulative holding probability at buyer arrivals. This characterization enables a unified analysis of the general setting and leads to sharper guarantees in several important special cases.

  For balanced trading (\(n=m\)), we improve the competitive ratio to \(\frac{2n^2-n}{(n+4^{-n}-1)(n+1)}\), which is strictly below \(2\) for every \(n>1\). For the single-seller case (\(m=1\)), we sharpen the analysis of the fixed-threshold algorithm to obtain a competitive ratio converging to $\frac{2e^2}{e^2+1}\approx1.76$. Furthermore, by a symmetry property of our general bound, the same guarantee also holds for the single-buyer case (\(n=1\)).

Finally, we exploit the additional structure of the single-seller setting to design an adaptive-threshold algorithm (Algorithm~\ref{alg2:single-seller}) whose threshold depends on the seller's arrival position. This algorithm remains \(2\)-competitive in the worst case, while its competitive ratio converges to $\frac{e}{e-1}\approx1.58$
as the number of buyers tends to infinity. A summary of our results is given in Table~\ref{tab:results}.

\begin{table}[H]
\caption{Results on the Prophet Trading Problem}\label{tab:results}
\centering
\begin{tabular}{ccc}
\hline
 Parameter & Competitive Ratio& Reference \\
\hline
 arbitrary $m,n$             &           $2$ &Algorithm~\ref{alg:general}, Theorem~\ref{thm:general}\\
  $m=n$             &                 $\frac{2n^2 - n}{(n + 4^{-n} - 1)(n + 1)}\le2$ & Algorithm~\ref{alg:general}, Theorem~\ref{thm:balanced}\\
 $m=1$                 &      $\frac{2n (n+3)^n}{(n-1)(n+3)^n+(n+1)^{n+1}}\le2$ & Algorithm~\ref{alg:general}, Theorem~\ref{thm:non-adaptive} \\
 $m=1$, $n\to\infty$ &  $\frac{e}{e-1}\approx 1.58$ & Algorithm~\ref{alg2:single-seller}, Theorem~\ref{thm:adaptive}\\
       $n=1$             &  $\frac{2m (m+3)^m}{(m-1)(m+3)^m+(m+1)^{m+1}}\le2$ & Algorithm~\ref{alg:general}, Corollary~\ref{cor:singlebuyer} \\
       $n=1$, $m\to\infty$ &  $\frac{2e^2}{e^2+1}\approx 1.76$ & Algorithm~\ref{alg:general}, Corollary~\ref{cor:singlebuyer}\\
$m=n=1$             &                 2 (tight)  &\cite{correa2026trading}, Theorem~\ref{thm:lowerbound}\\
\hline
\end{tabular}
\end{table}

\paragraph{Related work.}
Threshold algorithms play a central role in prophet inequalities
\citep{krengel1977semiamarts,krengel1978semiamarts,samuelcahn1984comparison}. In the classical single-item setting, a fixed-threshold policy accepts the first offer exceeding a prescribed threshold. Two standard choices, the median threshold and the mean threshold, both achieve the optimal competitive ratio of \(2\). Threshold policies also connect prophet inequalities to truthful posted-price mechanisms for sequential auctions; see \citet{lucier2017economic}.

Our work is most closely related to the trading prophets problem introduced by \citet{correa2026trading}. In their model, the decision-maker faces a sequence of market prices and may repeatedly buy and sell at these prices, holding at most one item at any time. Under i.i.d.\ known prices, they established the optimal competitive ratio of \(2\). For unknown i.i.d.\ prices with \(n\) observations, they obtained a competitive ratio of \((2n-2)/(n-2)\). For the non-i.i.d.\ setting, they proved that no effective algorithm exists in general. \citet{azar2026initial} showed that initial capital, in the form of one item, overcomes this impossibility with a competitive ratio of~3.  Our model with one buyer and one seller is essentially the same as the counterpart in \cite{correa2026trading}, up to a factor of two in the expected GFT. However, in general, our prophet-trading model differs in that buyers and sellers are explicitly distinguished--the trader can buy only from sellers and sell only to buyers. Consequently, a favorable price does not necessarily represent a feasible trading opportunity, and inventory dynamics depend on the realized sequence of buyer and seller arrivals. This buyer--seller asymmetry fundamentally changes both the prophet benchmark and the analysis of threshold algorithms. For example, a local minimum in the price sequence may occur at a buyer's arrival, and a local maximum at a seller's arrival. Hence, the local-extremum characterization of the offline optimum (buying at every local minimum and selling at every local maximum) in their model does not carry over to ours. In particular, the holding probability in the threshold algorithm of \citet{correa2026trading} is constant at $1/2$ at every step except the first and the last. By contrast, in the counterpart of our model (the balanced setting), this probability rises after processing a seller and falls after processing a buyer, making the analysis considerably more involved.

Another closely related line of work concerns secretary trading. The classical secretary problem
\citep{ferguson1989who,gilbert1966recognizing,lindley1961dynamic,bruss1984unified}
assumes no prior knowledge of the underlying distribution and relies only on the random arrival order of the candidates. Building on this framework, \citet{koutsoupias2018online} introduced the secretary trading problem and gave an online algorithm with competitive ratio \(1+O(n^{-1/3}\log n)\) for maximizing social welfare in balanced trading. Their benchmark is the sum of the \(m\) highest prices, which generally exceeds the maximum achievable social welfare under the online trading constraint. They also showed that no effective online algorithm exists for maximizing the trader's GFT. Subsequently, \citet{chen2024algorithms} generalized the algorithm for the balanced setting to large numbers of buyers and sellers and preserved the same asymptotic guarantee. \citet{chen2026online} completely resolved the single-seller case by proving the tight competitive ratio of \(4e^2/(e^2+1)\).

Both prophet inequalities and secretary problems have inspired numerous variants, including prophet secretary problems \citep{ehsani2024prophet,correa2021prophet}, ordered prophet inequalities \citep{peng2022order,bubna2023prophet}, i.i.d.\ prophet inequalities \citep{hill1982comparisons,correa2017posted}, matroid and multi-choice extensions \citep{kleinberg2005multiple,babaioff2007knapsack,babaioff2007matroids,babaioff2018matroid,hajiaghayi2007automated,alaei2014bayesian,arnosti2023tight,kleinberg2012matroid}, and, more recently, sample-driven prophet inequalities \citep{correa2022prophet,correa2023sample}.

\paragraph{Paper organization.}
The remainder of the paper is organized as follows. Section~\ref{sec:pre} introduces the formal model and the prophet benchmark. Section~\ref{sec:prop_b} studies the general prophet trading problem. Section~\ref{sec:prop_ref} investigates several important special cases and presents an adaptive-threshold algorithm for the single-seller setting. Finally, Section~\ref{sec:conclusion} concludes with directions for future research.

\section{Preliminaries}\label{sec:pre}
We now formalize the online trading model. All goods, hereafter called \emph{items}, are homogeneous. The market consists of $n$ buyers, $m$ sellers, and a trader. Buyers and sellers, collectively called \emph{agents}, arrive at the trader in a uniformly random order. All valuations (prices) are drawn independently from a common known nonnegative distribution with a finite first moment, and these draws are independent of the arrival order. The buyers' valuations for an item are denoted by $X_1,\ldots,X_n$, and the sellers' valuations by $X_{n+1},X_{n+2},\ldots,X_{n+m}$.  We index the agents according to their arrival order. The \emph{$i$-th agent} (respectively, the \emph{$i$-th buyer} and the \emph{$i$-th seller}) is the one whose arrival is preceded by exactly $i-1$ agents (respectively, buyers and sellers). We say that the $i$-th agent is \emph{earlier than}, or \emph{precedes}, the $j$-th agent if and only if $i<j$.

Upon arrival, the $i$-th agent reveals their valuation $X_{\sigma(i)}$ as the trading \emph{price} to the trader, where $\sigma\in S_{n+m}$ is a permutation of $[n+m]=\{1,2,\ldots,n+m\}$ mapping the arrival index $i$ to the identity $\sigma(i)$ of the arriving agent. In particular, the $i$-th agent is a buyer if and only if $\sigma(i)\le n$.

\begin{itemize}
\item Each seller initially owns exactly one item. Neither the buyers nor the trader initially owns any items, and each buyer demands exactly one item.
\item If the $i$-th agent is a buyer, they offer to buy one item from the trader at price $X_{\sigma(i)}$.
\item If the $i$-th agent is a seller, they offer to sell their item to the trader at price $X_{\sigma(i)}$.
\end{itemize}

For each arriving agent $\sigma(i)$, $i=1,\ldots,n+m$, the trader must make an immediate and irrevocable decision $Y_{\sigma(i)}\in\{0,1\}$, where $Y_{\sigma(i)}=1$ indicates that a trade takes place and $Y_{\sigma(i)}=0$ otherwise.
The trader may sell an item to a buyer only if the trader currently holds at least one item.
The trader's gain from trade (GFT) is defined as the difference between the revenue collected from buyers and the payments made to sellers:
\[
\sum_{i=1}^n X_iY_i-\sum_{i=n+1}^{n+m}X_iY_i.
\]
The objective of the online trading problem is to choose the trading decisions $Y_1,\ldots,Y_{n+m}$ so as to maximize the GFT.

\paragraph{Prophet trading.} Throughout the paper, we use the terms ``online algorithm'' and ``trader'' interchangeably. We refer to the offline optimal algorithm as the \emph{prophet}. Under the random-arrival assumption, the prophet knows the realized arrival order and all realized prices before trading begins; it computes an optimal trading strategy for that order and those prices. Since the arrival order is uniformly random, both the prophet's GFT and the online algorithm's GFT are random variables, with expectations taken jointly over the uniform arrival order, the i.i.d.\ prices, and any algorithmic randomness. The \emph{competitive ratio} of an online algorithm is defined as the worst-case ratio between the expected GFT achieved by the prophet and that achieved by the online algorithm. For brevity, we refer to the online trading model with this prophet benchmark as \emph{prophet trading}.

For any prophet trading algorithm $\talg$ and any given trading instance, we slightly abuse notation and let $\be[\talg]$ denote the expected GFT achieved by $\talg$. Similarly, we write $\be[\opt]$ for the prophet's expected GFT, where $\opt$ denotes the optimal offline algorithm.

\paragraph{Price distributions.} In this paper, we study prophet trading under the independent and identically distributed (i.i.d.) known-distribution model: each agent independently draws its price from \emph{the same known} distribution. We adopt this assumption for two reasons. First, \citet{koutsoupias2018online} proved that, when the distribution is unknown, no effective online algorithm {with a bounded competitive ratio} exists for maximizing the trader's expected GFT. Second, even if the buyer and seller distributions are known, allowing them to differ also precludes any bounded competitive ratio.

The following example adapted from \citet{correa2026trading} illustrates why the common-distribution assumption is necessary. Suppose there is one buyer and one seller. The buyer's price satisfies $X_1=0$ with probability $1-\epsilon$ and $X_1=1/\epsilon$ with probability $\epsilon$, while the seller's price is deterministically $X_2=1$. The prophet obtains a positive GFT whenever the seller arrives before the buyer and $X_1=1/\epsilon$, yielding an expected GFT of $(1-\epsilon)/2>0$. By contrast, an online trader must decide whether to purchase the seller's item when the seller arrives. Buying at price $1$ yields expected profit $\mathbb{E}[X_1]-1=0$, whereas declining the offer precludes any future trade. If the buyer arrives first, no trade is possible. Hence every online trader achieves expected GFT at most $0$, which is already guaranteed by trivially skipping all trades. In the i.i.d.\ known-distribution setting, the symmetry of the price distributions separates price randomness from arrival-order randomness and enables a clean characterization of the prophet's expected GFT under a uniformly random arrival order.

To handle price distributions with atoms, we employ an independent randomized tie-breaking rule to achieve any exact target quantile \(p \in (0, 1)\). Specifically, we select a price threshold \(T\) and a tie-breaking probability \(\gamma \in [0, 1]\) such that \(\mathbb{P}(X_1 < T) + \gamma\cdot \mathbb{P}(X_1 = T) = p\). During the algorithmic execution and analysis, if an agent's price exactly equals \(T\), we classify it as ``below the threshold'' with probability \(\gamma\), and ``above the threshold'' otherwise. This randomization is independent across all agents, prices, and arrival orders.

For ease of exposition, we slightly abuse notation. Throughout the paper, the expression \(X_i < T\) (or \(X_i \le T\)) refers to the \emph{event} that agent \(i\) is classified as below the threshold under this randomized rule, rather than a strict numerical comparison. Its complement is denoted by \(X_i > T\). Note that this tie-breaking mechanism is used solely for classification; the actual monetary prices and payoffs remain unchanged. For continuous (atomless) distributions, this convention naturally reduces to standard numerical comparisons.

\paragraph{Competitive ratio lower bound.} When $n=m=1$, our prophet trading problem is closely related to the {trading prophet problem} of \citet{correa2026trading},
where a trade occurs whenever the first agent's price is lower than the second's. In our prophet trading model, however, no trade is possible if the buyer arrives before the seller. Consequently, the
expected GFT in our setting is exactly half of that in their setting. Using a two-price instance, \citet{correa2026trading} proved a lower bound of 2 for the competitive ratio of their problem. The same lower bound carries over to our model. We state this lower bound formally and include the proof in Appendix~\ref{app:lowerbound} for completeness.
\begin{theorem}[\cite{correa2026trading}]\label{thm:lowerbound}
    When $n=m=1$, for any $\epsilon\in(0,1/2)$, there is an instance of prophet trading such that for any online algorithm $\talg$, it holds that $\be[\talg]\le \be[\opt]/(2-2\epsilon)$.
\end{theorem}

\section{General trading}\label{sec:prop_b}

This section studies prophet trading with $n$ buyers and $m$ sellers. We first derive an upper bound on the prophet's GFT, and then propose a threshold algorithm whose threshold depends on both $n$ and $m$. We prove a 2-competitive guarantee for arbitrary $n$ and $m$, matching the lower bound of 2 in Theorem~\ref{thm:lowerbound}.

\subsection{The prophet}\label{sec:opt}

We begin by establishing an upper bound on the prophet's GFT, which serves as the basis for our competitive ratio analysis.

\begin{lemma}\label{lem:opt-prophet}
    Let $T\in \mathbb{R}$ such that $\bp(X_1< T)=\rho=\frac{n+1}{n+m+2}$. Then
    \[
        \be [\opt]\le \frac{nm}{m+1}\cdot \be[X_1\cdot \bo_{X_1>T}]
        -\frac{nm}{n+1}\cdot \be[X_1\cdot \bo_{X_1<T}].
    \]
\end{lemma}
\begin{proof}
    To upper bound the prophet's value, we relax the feasible action set by introducing a hypothetical charity with an unlimited supply of items.
    \begin{enumerate}
        \item[\textbullet] When a seller arrives, provided they are followed by at least one buyer, the prophet \emph{can} buy the item from the seller and immediately sell it to the charity at price $T$.
        \item[\textbullet] When a buyer arrives, provided they are preceded by at least one seller, the prophet \emph{can} buy an item from the charity at price $T$ and immediately sell it to the buyer.
    \end{enumerate}
    The optimal value in this relaxed problem is at least $\be[\opt]$. Moreover, there is an optimal strategy for the relaxed problem under which the prophet buys from every eligible seller with price below $T$ and sells to every eligible buyer with price above $T$.

Indeed, fix an arrival order and a realization of all prices. Consider any feasible strategy in the relaxed problem that uses a direct trade between a seller with price $s$ and a later buyer with price $b$. The contribution of this trade to the GFT is $b-s$. Since this seller is followed by that buyer, the seller is eligible to trade with the charity; similarly, since this buyer is preceded by that seller, the buyer is eligible to trade with the charity. We may therefore replace the direct trade by two charity trades: buy from the seller and sell the item to the charity at price $T$, and buy an item from the charity at price $T$ and sell it to the buyer. The total contribution becomes $(T-s)+(b-T)=b-s$, so the replacement does not change the GFT. Repeating this replacement for every direct seller-to-buyer trade, there is an optimal relaxed strategy that only trades with the charity. Once all trades are routed through the charity, the decisions are independent across agents, because the charity has unlimited supply and demand. For an eligible seller with price $x$, trading with the charity contributes $T-x$, while skipping the seller contributes $0$; hence it is optimal to buy exactly when $x<T$. Similarly, for an eligible buyer with price $x$, it is optimal to sell exactly when $x>T$. When $x=T$, either action contributes $0=x-T=T-x$; so randomized tie-breaking for price distributions with atoms (if any) leaves this argument unchanged.

    A buyer is preceded by at least one seller with probability $m/(m+1)$, and a seller is followed by at least one buyer with probability $n/(n+1)$. Hence
    \begin{align*}
        \be[\opt]&\le \frac{nm}{m+1} \cdot \be[(X_1-T)\cdot \bo_{X_1>T}]
        +\frac{nm}{n+1}\cdot \be[(T-X_1)\cdot \bo_{X_1<T}]\\
        &=\frac{nm}{m+1}\cdot \be[X_1\cdot \bo_{X_1>T}]
        -\frac{nm}{n+1}\cdot \be[X_1\cdot \bo_{X_1<T}]
        +T\cdot nm\cdot \left(\frac{\rho-1}{m+1}+\frac{\rho}{n+1}\right)\\
        &=\frac{nm}{m+1}\cdot \be[X_1\cdot \bo_{X_1>T}]
        -\frac{nm}{n+1}\cdot \be[X_1\cdot \bo_{X_1<T}],
    \end{align*}
    where the last equality follows from $\rho=(n+1)/(n+m+2)$.
\end{proof}

\subsection{The threshold algorithm}\label{sec:alg-general}

Our algorithm follows the standard threshold strategy: given a threshold $T\ge 0$, the trader purchases an item from a seller only if the observed price is below $T$, and sells an item to a buyer only if the observed price is above $T$.
In addition, \emph{our algorithm} imposes the extra rule that the trader purchases only when holding no item. This restriction is introduced solely for the algorithm and its analysis, and is \emph{not} part of the problem definition. Finally, to avoid ending with unsold inventory, if the trader still holds an item when the last buyer arrives, it sells the item regardless of the buyer's price.

The key is to choose an appropriate threshold $T$ that yields the best provable competitive ratio. We do so by selecting $T$ such that
 \[\bp(X_1<T)=\frac{n+1}{n+m+2}.\]
The algorithm maintains two variables: $h\in\{0,1\}$ indicates whether the trader currently holds an item, and $c$ records the number of buyers encountered so far. Whenever the current agent is a buyer, the algorithm first increments $c$. The pseudocode is given in Algorithm~\ref{alg:general}.

\begin{algorithm}[ht]
    \caption{\sc General prophet trading}\label{alg:general}
    Set $T\in\mathbb{R}$ such that $\bp(X_1<T)=({n+1})/({n+m+2})$\;
    $h\leftarrow 0$, $c\leftarrow 0$\;
    \For{$i=1$ \KwTo $n+m$}{
        \uIf{\em the $i$-th agent is a seller}{
            \If{\em $h=0$ {and} $c<n$ {and} $X_{\sigma(i)}<T$}{
                buy an item from the $i$-th agent\;
                $h\leftarrow 1$\;
            }
        }
        \Else(\tcp*[h]{the $i$-th agent is a buyer}){
            $c\leftarrow c+1$\;
            \If{\em $h=1$ {and} $\bigl(X_{\sigma(i)}>T$ {or} $c=n\bigr)$}{
                sell the item to the $i$-th agent\;
                $h\leftarrow 0$\;
            }
        }
    }
\end{algorithm}

  Conditional on a realized arrival order \(\sigma\), let \(B_{\sigma}\) ($\subset[n+m]$)  denote the set of arrival indices occupied by buyers. For each $i\in[n+m]$, let \(p_i^{\sigma}\) denote the probability that the trader holds an item immediately before processing the \(i\)-th arriving agent. The following lemma expresses the algorithm's expected GFT in terms of $\sum_{i\in B_{\sigma}}p^{\sigma}_{i}$, the cumulative holding probability over buyer arrivals.

\begin{lemma}\label{lma:alg-general}
    The GFT of Algorithm~\ref{alg:general} is
    \[
        \be[\mathrm{ALG}]=\be\left[\sum_{i\in B_{\sigma}}p^{\sigma}_{i}\right]
        \cdot \left(\be[X_1\cdot \bo_{X_1>T}]
        -\frac{m+1}{n+1}\cdot \be[X_1\cdot \bo_{X_1<T}]\right).
    \]
\end{lemma}

In the above identity, the expectation $\be\left[\sum_{i\in B_{\sigma}}p^{\sigma}_{i}\right]$ is taken over all $\sigma\in S_{n+m}$, that is
\[
    \be\left[\sum_{i\in B_{\sigma}}p^{\sigma}_{i}\right]
    =
    \frac{1}{(n+m)!}\sum_{\sigma\in S_{n+m}}\sum_{i\in B_{\sigma}}p_i^\sigma.
\]
The technical proof of Lemma~\ref{lma:alg-general} is deferred to Appendix~\ref{app:proof-alg-general}. The remaining technical point is to lower-bound the \emph{key inventory term} \(\be\left[\sum_{i\in B_{\sigma}}p^{\sigma}_{i}\right]\).  For convenience, we write
\[
\rho:=\frac{n+1}{n+m+2} \quad\text{and}\quad
\bar\rho:=1-\rho=\frac{m+1}{n+m+2}.
\]

\begin{lemma}\label{lma:sump-integral}
    Let
    \[
        K_{n,m}:=\int_0^1(1-x)(1-\rho x)^{m-1}(1-\bar\rho x)^{n-1}\dx.
    \]
    Then
    \[
        \be\left[\sum_{i\in B_{\sigma}}p^{\sigma}_{i}\right]=nm\rho K_{n,m}\ge \frac{nm}{2(m+1)}.
    \]
\end{lemma}
The proof of the technical lemma is deferred to Appendix~\ref{app:proof-holding-core}. Although the following proof of the competitive ratio of 2 for Algorithm~\ref{alg:general} requires only the inequality $\be\left[\sum_{i\in B_{\sigma}}p^{\sigma}_{i}\right]\ge \frac{nm}{2(m+1)}$, the exact identity $\be\left[\sum_{i\in B_{\sigma}}p^{\sigma}_{i}\right]=nm\rho K_{n,m}$ in the lemma will play a crucial role in the next section to derive improved competitive ratios for special cases of prophet trading.

\begin{theorem}\label{thm:general}
    Algorithm~\ref{alg:general} has competitive ratio at most
    \[
        R_{n,m}:=\frac{1}{(m+1)\rho K_{n,m}}
        =\frac{n+m+2}{(n+1)(m+1)K_{n,m}}.
    \]
    In particular, \(R_{n,m}\le2\).
\end{theorem}
\begin{proof}
    Setting
    \(
        \Delta:=\be[X_1\cdot\bo_{X_1>T}]
        -\frac{m+1}{n+1}\be[X_1\cdot\bo_{X_1<T}]
    \)
    gives $ \be[\opt]\le \frac{nm}{m+1}\Delta$ by Lemma~\ref{lem:opt-prophet}. Lemmas~\ref{lma:alg-general} and \ref{lma:sump-integral} give
    \(
        \be[\mathrm{ALG}]=nm\rho K_{n,m}\Delta.
    \)
    Hence
    \[
        \frac{\be[\opt]}{\be[\mathrm{ALG}]}
        \le \frac{1}{(m+1)\rho K_{n,m}}
        = R_{n,m}.
    \]
    Finally, Lemma~\ref{lma:sump-integral} gives \(K_{n,m}\ge1/(2\rho(m+1))\), so \(R_{n,m}\le2\).
\end{proof}

Although sellers and buyers play asymmetric roles in prophet trading, our algorithm exhibits a remarkable symmetry with respect to the numbers of sellers and buyers, as stated in the following corollary.

\begin{corollary}\label{cor:symmetry}
    The upper bound \(R_{n,m}\) on the competitive ratio of Algorithm \ref{alg:general} is symmetric in \(n\) and \(m\), i.e., $R_{n,m}=R_{m,n}$ for all $n,m\ge1$.
\end{corollary}
\begin{proof}
 To see the symmetry, we write
    \[
        \rho_{a,b}:=\frac{a+1}{a+b+2}
        \quad\text{and}\quad
        \bar\rho_{a,b}:=\frac{b+1}{a+b+2}.
    \]
    Then \(\rho_{m,n}=\bar\rho_{n,m}\) and \(\bar\rho_{m,n}=\rho_{n,m}\); so
    \begin{align*}
        K_{m,n}
        &=\int_0^1(1-x)(1-\rho_{m,n}x)^{n-1}(1-\bar\rho_{m,n}x)^{m-1}\dx\\
        &=\int_0^1(1-x)(1-\bar\rho_{n,m}x)^{n-1}(1-\rho_{n,m}x)^{m-1}\dx\\
        &=K_{n,m}.
    \end{align*}
    Moreover, \((m+1)\rho_{n,m}=(n+1)\rho_{m,n}\). By Theorem~\ref{thm:general}, we have \(R_{n,m}=R_{m,n}\).
\end{proof}

\section{Special cases}\label{sec:prop_ref}
In this section, we investigate two special cases of prophet trading in which the general analysis can be sharpened. The balanced case ($n=m$) is studied in Section~\ref{sec:balanced}, while the single-seller case ($m=1$) and the single-buyer case ($n=1$) are considered in Section~\ref{sec:singleseller} and Section~\ref{sec:singlebuyer}, respectively. Although the threshold algorithm developed in the previous section (Algorithm~\ref{alg:general}) already achieves a 2-competitive ratio for arbitrary numbers of buyers and sellers, these special settings admit stronger results. Depending on the additional structure available, this is achieved either through a more refined analysis of the general threshold algorithm or through the design of a specialized threshold rule.

\subsection{Balanced trading}\label{sec:balanced}
A prophet trading instance is called \emph{balanced} if the numbers of buyers and sellers are equal, i.e., $n=m$. Under the general framework established in Section~\ref{sec:prop_b}, the refined bound in Theorem~\ref{thm:general} reduces the analysis to evaluating \(K_{n,n}\). Figure~\ref{fig:balanced-ratio} plots this refined upper bound.

\begin{theorem}\label{thm:balanced}
    For balanced instances of prophet trading, Algorithm~\ref{alg:general} has competitive ratio at most
    \[
        \frac{2n^2-n}{(n+4^{-n}-1)(n+1)}\le 2.
    \]
    Furthermore, this upper bound tends to \(2\) as \(n\to\infty\).
\end{theorem}
\begin{proof}
    When \(m=n\), we have \(\rho=\bar\rho=1/2\). By Theorem~\ref{thm:general}, the competitive ratio of Algorithm~\ref{alg:general} is at most
    \[
        R_{n,n}=\frac{2}{(n+1)K_{n,n}},
    \]
    where
    \[
        K_{n,n}=\int_0^1(1-x)\left(1-\frac{x}{2}\right)^{2n-2}\dx.
    \]
    Substituting \(y=1-x/2\) gives
    \begin{align*}
        K_{n,n}=2\int_{\frac12}^{1}(2y-1)y^{2n-2}\diff y=\frac{2(n+4^{-n}-1)}{n(2n-1)}.
    \end{align*}
    Therefore
    \[
        R_{n,n}
        =\frac{2n^2-n}{(n+4^{-n}-1)(n+1)}.
    \]
    The inequality \(R_{n,n}\le2\) is implied by Theorem~\ref{thm:general}. Finally,
    \[
        \lim_{n\to\infty}R_{n,n}
        =
        \lim_{n\to\infty}
        \frac{2-\frac{1}{n}}{\left(1+\frac{4^{-n}-1}{n}\right)\left(1+\frac{1}{n}\right)}
        =2.\qedhere
    \]
\end{proof}

\begin{figure}[t]
    \centering
    \includegraphics[width=0.9\textwidth]{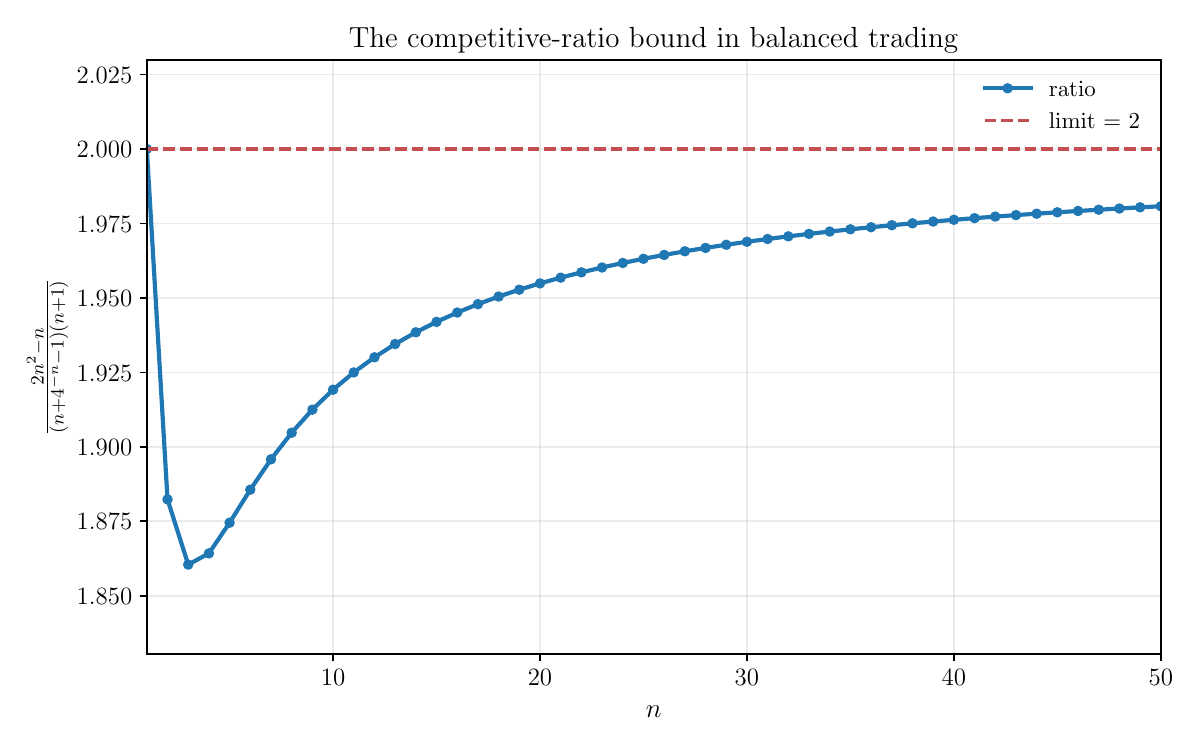}
    \caption{The competitive-ratio bound in balanced trading.}
    \label{fig:balanced-ratio}
\end{figure}

\subsection{Single-seller trading}\label{sec:singleseller}
For prophet trading with a single seller, we devote two subsections to the study of fixed-threshold and adaptive-threshold algorithms, respectively.

\subsubsection{Fixed threshold}
When there is only one seller, Algorithm~\ref{alg:general} works with a fixed threshold quantile \(\rho=(n+1)/(n+3)\). The refined bound in Theorem~\ref{thm:general} reduces the analysis to evaluating \(K_{n,1}\).

\begin{theorem}\label{thm:non-adaptive}
    In the single-seller case, Algorithm~\ref{alg:general} has competitive ratio at most
    \[\frac{2n (n+3)^n}{(n-1)(n+3)^n+(n+1)^{n+1}}.\]
    Furthermore, this upper bound tends to $(2e^2)/(e^2+1)\approx 1.76$ as \(n\to\infty\).
\end{theorem}
\begin{proof}
For \(m=1\), we have \(\rho=(n+1)/(n+3)\) and \(\bar\rho=2/(n+3)\). By Theorem~\ref{thm:general}, the competitive ratio of Algorithm~\ref{alg:general} is at most
\[
    R_{n,1}=\frac{1}{2\rho K_{n,1}},
\]
where
\[
    K_{n,1}=\int_0^1(1-x)(1-\bar\rho x)^{n-1}\dx .
\]
The substitution \(y=1-\bar\rho x\) gives
\[
    K_{n,1}=\frac{1}{\bar\rho^2}\int_{\rho}^{1}(y-\rho)y^{n-1}\diff y
    =\frac{1}{\bar\rho^2}\left(\frac{1-\rho^{n+1}}{n+1}-\frac{\rho(1-\rho^n)}{n}\right).
\]
Equivalently,
\[
    n\rho K_{n,1}
    =\frac{n-1}{4}+\frac{n+3}{4}\left(\frac{n+1}{n+3}\right)^{n+1}.
\]
Therefore
\begin{align}
    R_{n,1} &= \frac{2n}{n-1+(n+3)\left(\frac{n+1}{n+3}\right)^{n+1}} \label{eq:singleseller}\\
     &=\frac{2n (n+3)^n}{(n-1)(n+3)^n+(n+1)^{n+1}}.\nonumber
\end{align}
Dividing the denominator on the right-hand side of (\ref{eq:singleseller}) by \(n\) yields
\[
    \lim_{n\to\infty}\left(\frac{n-1}{n}+\frac{n+3}{n}\left(\frac{n+1}{n+3}\right)^{n+1}\right)
    = 1+\frac{1}{e^2}.
\]
Hence the upper bound on the competitive ratio tends to \(2/(1+e^{-2})=2e^2/(e^2+1)\).
\end{proof}

\subsubsection{Adaptive threshold}

Buyers who arrive before the single seller cannot participate in any feasible trade, so their prices do not affect the offline optimum. We therefore set the adaptive threshold as a function of the number of buyers preceded by the seller.  Recall that the identity label of the seller is $n+1$.

\begin{algorithm}[H]
    \caption{\sc Adaptive single-seller prophet trading}\label{alg2:single-seller}
 $k\leftarrow$ the arrival order of the seller, i.e., $\sigma^{-1}(n+1)$\;
 \If{$k=n+1$}{\Return\;}
 Set $T$ such that $\bp(X_1\le T)=({n+1-k})/({n+2-k})$\;
    \If{\em $X_{n+1}<T$}{
        buy an item from the seller\;
        \For{$i=k+1$ \KwTo $n$}{
            \If{$X_{\sigma(i)}>T$}{
                sell the item to the $i$-th agent\;
                skip the remaining agents\;
                \Return\;
            }
        }
        sell the item to the $(n+1)$-th agent\;
    }
    \Else{
        skip all agents\;
    }
\end{algorithm}

Given $k$ as defined in Algorithm~\ref{alg2:single-seller},
the first $k-1$ buyers are irrelevant. Let $q=n-k+1$ be the number of buyers arriving after the seller. Conditional on the seller's arrival position, the relevant instance therefore consists of exactly one seller and the $q=n-k+1$ remaining buyers. Since all prices are drawn independently from the same distribution, these $q+1$ prices are symmetric as random variables. The algorithm uses the common threshold $T$ such that $\mathbb{P}(X_1\le T)=q/(q+1)$, which mirrors the standard threshold choice in the i.i.d.\ prophet-inequality setting. Nevertheless, the truncated trading instance is not identical to the classical i.i.d.\ prophet inequality counterpart, because the objective is to maximize the GFT rather than the selected value itself.

We use $\be[\talg_k]$ and $\be[\opt_k]$ to denote the expected GFTs of Algorithm~\ref{alg2:single-seller} and the prophet, respectively, conditional on the seller being the $k$-th arriving agent. Clearly,
\begin{equation}\label{eq:sum}
     \be[\opt]=\frac{1}{n+1}\sum_{k=1}^{n+1}\be[\opt_{k}]
    \quad\text{and}\quad
        \be[\talg]=\frac{1}{n+1}\sum_{k=1}^{n+1}\be[\talg_{k}]
\end{equation}
\begin{lemma}\label{lem:nn}
  For any $k\in[n]$, let $q=n+1-k$. Then
  \[\frac{\be[\opt_k]}{\be[\talg_k]}\le \frac{(q+1)^q}{(q+1)^q-q^q}.\]
  Furthermore, as $q$ increases, the upper bound on the ratio decreases monotonically from $2$ to $e/(e-1)\approx 1.58$.
\end{lemma}
\begin{proof}
Note that trading is effectively conducted on a market of \(q+1\) agents: one seller and \(q\) buyers, whose prices are revealed as \(X_{\sigma(k)}=X_{n+1},X_{\sigma(k+1)},\ldots,X_{\sigma(n+1)}\). For notational convenience, set \(X'_i=X_{\sigma(k+i)}\) for \(i=1,2,\ldots,q\) and \(X'_{q+1}=X_{n+1}\). We first claim that
\[(q+1)\cdot \be[X'_1\cdot \bo_{X'_1>T}]\ge \be[\max\{X'_1,\dots,X'_{q+1}\}].\]
Indeed, since
\begin{align*}
    \max\{X'_1,\dots,X'_{q+1}\}&\le \max\{X'_1,\dots,X'_{q+1},T\}\\
    &=T+\max\{X'_1-T,\dots,X'_{q+1}-T,0\}\\
    &\le T+\sum_{i=1}^{q+1} \max\{X'_i-T,0\}\\
    &= T+\sum_{i=1}^{q+1} (X'_i-T)\cdot \bo_{X'_i>T},
\end{align*}
we have
\begin{align*}
    \be[\max\{X'_1,\dots,X'_{q+1}\}]&\le T+(q+1)\cdot (\be[X'_1\cdot \bo_{X'_1>T}]-T\cdot \bp(X'_1> T))\\
    &=T+(q+1)\cdot \be[X'_1\cdot \bo_{X'_1>T}]-(q+1)\cdot T\cdot \frac{1}{q+1}\\
    &=(q+1)\cdot \be[X'_1\cdot \bo_{X'_1>T}].
\end{align*}
Thus,
\begin{align*}
    \be[\opt_k]&=\be[\max\{X'_1,\dots,X'_{q+1}\}]-\be[X'_{q+1}]\\
    &\le (q+1)\cdot \be[X'_1\cdot \bo_{X'_1>T}]-\be[X'_1]\\
    &=q\cdot \be[X'_1\cdot \bo_{X'_1>T}]-\be[X'_1 \cdot \bo_{X'_1<T}].
\end{align*}

To compute the algorithm's expected GFT, we use an equivalent interpretation in which all agents are treated as buyers. Immediately before the seller arrives, suppose the trader is forced to acquire an item at cost \(X_{n+1}\). At the seller's arrival, treat the seller as a buyer. If this agent's price, i.e., \(X_{\sigma(k)}=X_{n+1}=X'_{q+1}\), is below the threshold \(T\), then the trader keeps the item; this corresponds to buying from the seller in Algorithm~\ref{alg2:single-seller}. Otherwise, the trader immediately transfers the item back at price \(X_{n+1}\); this corresponds to not buying from the seller. We also note that
\begin{itemize}
    \item The threshold \(T\) satisfies \(\bp(X_i\le T)=q/(q+1)\) for all \(i\).
    \item Whether the trader holds the item immediately before the \(i\)-th agent arrival is independent of \(X_{\sigma(i)}\), and the probability of holding the item is \(\left(q/(q+1)\right)^{i-k}\).
\end{itemize}
These observations, along with the i.i.d.\ distributional assumption, give
\begin{align*}
    \be[\talg_k]&=-\be[X'_{q+1}]+\sum_{i=1}^{q} \left(\frac{q}{q+1}\right)^{i-1} \cdot \be[X'_i\cdot \bo_{X'_i>T}]+\left(\frac{q}{q+1}\right)^q \be[X'_q]\\
    &=\left(1-\left(\frac{q}{q+1}\right)^q\right)\cdot ((q+1)\cdot \be[X'_1\cdot \bo_{X'_1>T}]-\be[X'_1])\\
    &=\left(1-\left(\frac{q}{q+1}\right)^q\right)\cdot (q\cdot \be[X'_1\cdot \bo_{X'_1>T}]-\be[X'_1 \cdot \bo_{X'_1<T}]).
\end{align*}
Thus, we derive
\begin{equation}
    \frac{\be[\opt_k]}{\be[\talg_k]}\le \frac{(q+1)^q}{(q+1)^q-q^q}.\label{eq:ratio}
\end{equation}
Finally, observe that the sequence \(a_q=(1+1/q)^q\) increases from \(2\) to \(e\). Since the ratio on the right-hand side of (\ref{eq:ratio}) is \(a_q/(a_q-1)\), it follows that this bound decreases from \(2\) to \(e/(e-1)\).
\end{proof}

Building on Lemma~\ref{lem:nn}, we can now evaluate the overall performance of Algorithm~\ref{alg2:single-seller}.
\begin{theorem}\label{thm:adaptive}
    Algorithm~\ref{alg2:single-seller} is a $2$-competitive algorithm for the single-seller prophet trading problem. More precisely,
    \[
        \frac{\be[\opt]}{\be[\talg]}\le\beta_n^{-1},\quad
       \text{where }\; \beta_n:=\frac1n\sum_{q=1}^{n}\left(1-\left(\frac{q}{q+1}\right)^q\right).
    \]
    The competitive-ratio bound $\beta_n^{-1}\le2$ approaches $e/(e-1)\approx1.58$ as $n\to\infty$.
\end{theorem}
\begin{proof}
Under the uniform arrival assumption, the number \(q=n+1-k\) of buyers remaining after the seller arrives is uniform over \(\{0,1,\dots,n\}\). The case \(q=0\) contributes zero to any GFT. Reindexing by \(q=n+1-k\), from (\ref{eq:sum}) we derive
\[
    \be[\opt]=\frac{1}{n+1}\sum_{q=1}^{n}\be[\opt_{n+1-q}]
    \quad\text{and}\quad
        \be[\talg]=\frac{1}{n+1}\sum_{q=1}^{n}\be[\talg_{n+1-q}].
\]
For $q=1,\ldots,n$, let
\[
    V_q:=\be[\opt_{n+1-q}]\quad\text{and}\quad c_q:=1-\left(\frac{q}{q+1}\right)^q.
\]
Adding one more buyer while keeping the seller price and the existing buyer prices unchanged cannot decrease the optimal GFT, so $V_q$ is nondecreasing in $q$. The sequence $c_q$ is also nondecreasing by Lemma~\ref{lem:nn}, whose calculation gives $\be[\talg_{n+1-q}]\ge c_qV_q$. Chebyshev's sum inequality therefore applies to these two nondecreasing sequences, giving
\[
    n\sum_{q=1}^n c_qV_q-
    \left(\sum_{q=1}^n c_q\right)\left(\sum_{q=1}^n V_q\right)
    =\sum_{1\le i<j\le n}(c_j-c_i)(V_j-V_i)\ge0.
\]
Averaging over the seller's arrival position and applying the above inequality yields
\[
    \be[\talg]\ge\frac1{n+1}\sum_{q=1}^n c_qV_q
    \ge\frac{\beta_n}{n+1}\sum_{q=1}^n V_q
    =\beta_n\be[\opt].
\]
Since $c_1=1/2$ and $c_q\to1-1/e$, we have $\beta_n^{-1}\le2$ and $\beta_n^{-1}\to e/(e-1)$.
\end{proof}

\subsection{Single-buyer trading}\label{sec:singlebuyer}
 By the symmetry of the bound \(R_{n,m}\) established in Corollary~\ref{cor:symmetry}, together with a symmetric, purely notational adaptation of the proof of Theorem~\ref{thm:non-adaptive}, we obtain the following bound for the single-buyer case.
\begin{corollary}\label{cor:singlebuyer}
   When \(n=1\), Algorithm~\ref{alg:general} achieves a competitive ratio of at most
    \[
        \frac{2m (m+3)^m}{(m-1)(m+3)^m+(m+1)^{m+1}},
    \]
    which converges to $(2e^2)/(e^2+1)\approx 1.76$ as \(m\to\infty\).
\end{corollary}

Unlike Theorem~\ref{thm:non-adaptive}, whose result extends symmetrically to the single-buyer case (Corollary~\ref{cor:singlebuyer}), the asymptotic 1.58-competitive guarantee established in Theorem~\ref{thm:adaptive} for the single-seller case does not admit such a direct extension. This is because the roles of sellers and buyers are inherently asymmetric in prophet trading, and the adaptive threshold in Algorithm~\ref{alg2:single-seller} depends explicitly on the seller's arrival position.

\section{Conclusion}\label{sec:conclusion}

We have studied online trading under i.i.d.\ known price distributions, using the prophet's gain from trade as the benchmark. In the general setting with $n$ buyers and $m$ sellers,
the threshold, which upper bounds the i.i.d.\ price with probability $(n+1)/(n+m+2)$, balances the buyer-side and seller-side contributions and gives an algorithm with competitive ratio at most 2 for all $n$ and $m$. The same framework yields sharper guarantees in two special cases. In the balanced case $n=m$, evaluating the key inventory term exactly improves the ratio to $(2n^2-n)/((n+4^{-n}-1)(n+1))$. In the single-seller case ($m=1$) or the single-buyer case ($n=1$), the fixed-threshold specialization has competitive ratio $2n (n+3)^n/((n-1)(n+3)^n+(n+1)^{n+1})$ or $2m (m+3)^m/((m-1)(m+3)^m+(m+1)^{m+1})$; both ratios converge to \(2e^2/(e^2+1)\) as the corresponding parameter tends to infinity. Finally, an adaptive threshold rule for the single-seller case remains 2-competitive in the worst case and improves the asymptotic ratio to \(e/(e-1)\).

Several natural directions remain open for future work. The first is to design adaptive algorithms, whose thresholds or inventory decisions depend on the realized arrival history rather than only on the parameters $n$ and $m$. A natural starting point is the single-buyer setting, with the general setting as the ultimate goal. The second is to study sample-driven i.i.d.\ prophet trading, where the trader has access only to historical samples from the common price distribution instead of the distribution itself. Further questions include whether a constant competitive ratio is achievable for non-i.i.d.\ prices when the trader starts with an item, or when the buyer distribution first-order stochastically dominates the seller distribution, and whether matching lower bounds can be established for our special-case guarantees.

\paragraph{Acknowledgements.}
This work was supported by the National Natural Science Foundation of China (Grant Nos. 11971046, 12331014, 12601619, 12671378 and 71871009), the Science and Technology Commission of Shanghai Municipality (Grant No. 22DZ2229014), and the Fundamental Research Funds for the Central Universities.

\bibliographystyle{elsarticle-num-names-alpha}
\bibliography{ref}

\clearpage
\appendix

\section{Lower bound proof}\label{app:lowerbound}
\begin{proof}[Proof of Theorem \ref{thm:lowerbound}]
    Fix $\epsilon\in(0,1/2)$ and consider the instance
    \[
        X_i = \begin{cases}
        2 & \text{w.p.~$\epsilon$},\\
        1 & \text{w.p.~$1 - 2\epsilon$},\\
        0 &\text{w.p.~$\epsilon$}
        \end{cases} \qquad \text{for $i \in \{1,2\}$}
    \]
   where the setting is restricted to exactly one buyer and one seller. The probability that the seller arrives first and the buyer second is $1/2$. Therefore, the prophet's expected GFT is
    \[
        \be[\opt] = \frac{1}{2} \cdot (2\bp(X_2=0)\bp(X_1=2)+\bp(X_2=0)\bp(X_1=1)+\bp(X_2=1)\bp(X_1=2))=\epsilon-\epsilon^2.
    \]
    Now consider the optimal online algorithm $\talg$. A trade can occur only if the seller arrives first and the buyer second, which has probability $1/2$. In this case, if the seller's price is 0, the optimal decision is to buy; since the buyer's expected price is 1, the expected GFT is 1. If the seller's price is 2, the optimal decision is not to buy, giving expected GFT 0. If the seller's price is 1, then the expected GFT is 0 regardless of whether $\talg$ buys, because the buyer's expected price is 1. Therefore,
    $\be[\talg]=\epsilon/2$,
    giving $\be[\opt] = (2-2\epsilon)\be[\talg]$ as claimed.
\end{proof}
\section{Proof of Lemma~\ref{lma:alg-general}}\label{app:proof-alg-general}
Recall that \(B_\sigma\) is the set of arrival indices occupied by buyers under the arrival order \(\sigma\). For convenience in the proof, write \(B_\sigma=\{b_1,\ldots,b_n\}\) with \(b_1<\cdots<b_n\). Since the trader initially holds no items, the initial holding probability is $p^{\sigma}_{1}=0$. For any step $i \in [b_n]$, the transition of the holding probability conditional on $\sigma$ satisfies the following recurrence:
\[
p^{\sigma}_{i+1}=\begin{cases}
    p^{\sigma}_{i}+\rho(1-p^{\sigma}_{i}) &\text{if $i\notin B_{\sigma}$},\\
    \rho p^{\sigma}_{i} &\text{if $i\in B_{\sigma}$ and $i\neq b_n$},\\
    0 &\text{if $i=b_n$}.
\end{cases}
\]
This recurrence relation arises directly from the trading decisions mandated by Algorithm~\ref{alg:general}:
\begin{itemize}
    \item \textbf{Seller Arrivals ($i \notin B_\sigma$):} The $(i+1)$-th step begins with the trader holding an item if they already held one at the $i$-th step (with probability $p^\sigma_i$, in which case they cannot buy from the current seller), or if they were empty-handed (with probability $1-p^\sigma_i$) and successfully purchased from the seller. A purchase occurs when the seller's price is below $T$, which happens with probability $\bp(X_1 < T) = \rho$. Combining these yields $p^{\sigma}_{i+1}= p^\sigma_i + \rho(1-p^\sigma_i)$.
    \item \textbf{Intermediate Buyer Arrivals ($i \in B_\sigma$ and $i \neq b_n$):} The trader holds an item at the next step if and only if they held one at the current step (with probability $p^\sigma_i$) and the item was not sold to the buyer. According to the algorithm, a sale is triggered if the buyer's price exceeds $T$ (probability $\bar\rho$). Thus, the item is retained only if the price falls below $T$ (probability $\rho$), leading to the transition $p^{\sigma}_{i+1}=\rho p^\sigma_i$.
    \item \textbf{The Last Buyer Arrival ($i = b_n$):} When the last buyer arrives, the condition $c=n$ is satisfied. Algorithm~\ref{alg:general} specifies that if the trader holds an item, it is unconditionally sold to this buyer to clear the inventory, regardless of the price. If the trader is already empty-handed, no trade occurs. In either case, the trader is guaranteed to hold no items immediately after this step, forcing $p^\sigma_{i+1} = 0$.
\end{itemize}

From the first case of the recurrence, when a seller arrives, we can rewrite the relation as $1-p^\sigma_{i+1}=\bar\rho(1-p^\sigma_i)$ where $\bar\rho = 1-\rho$. By unrolling this multiplicative relation from the initial condition $p^\sigma_1 = 0$, it immediately follows that for $1\le i\le b_1$:
\begin{equation}
    p^{\sigma}_{i}=1-\bar\rho^{i-1}. \label{eq:general-first}
\end{equation}
Similarly, for each buyer $b_j$ where $j\in [n-1]$, the recurrence implies $p^{\sigma}_{b_j+1}=\rho p^{\sigma}_{b_j}$. For the subsequent steps prior to the next buyer's arrival, i.e., $b_j+1\le i\le b_{j+1}$, the trader only encounters sellers. Inductively applying the seller transition relation yields:
\begin{equation}
    p^{\sigma}_{i}=1-\bar\rho^{i-b_j-1}\left(1-\rho p^{\sigma}_{b_j}\right). \label{eq:general-j}
\end{equation}

With these explicit expressions for the holding probabilities in hand, we are now ready to establish a crucial conservation identity for the inventory process. To this end, let $S_\sigma = [b_n] \setminus B_\sigma$ denote the set of sellers who arrive before the last buyer. As stated in Lemma~\ref{lma:inventory-balance}, this identity equates the algorithm's expected sales behavior (the LHS) with its expected purchase behavior (the RHS), thereby connecting the holding probabilities across both sides of the market.

\begin{lemma}\label{lma:inventory-balance}
    For each arrival order $\sigma$,
    \[
        \rho p^{\sigma}_{b_n}+\bar\rho\sum_{i\in B_{\sigma}}p^{\sigma}_{i}
        =\rho\sum_{i\in S_{\sigma}}(1-p^{\sigma}_{i}).
    \]
\end{lemma}
\begin{proof}
    For a fixed arrival order, we prove the identity by induction on the number of buyers processed. Let us first establish the base case for $n=1$. In a market with a single buyer arriving at index $b_1$, the set of buyers is $B_\sigma = \{b_1\}$ and the set of preceding sellers is $S_\sigma = \{1, \dots, b_1 - 1\}$. Substituting these sets into the claimed identity, the left-hand side reduces to:
    \[
        (\rho + \bar\rho) p^{\sigma}_{b_1} = p^{\sigma}_{b_1} = 1 - \bar\rho^{b_1 - 1},
    \]
    where the last step follows directly from \eqref{eq:general-first}. On the right-hand side, since $1 - p^\sigma_i = \bar\rho^{i-1}$ for all $1 \le i \le b_1 - 1$, evaluating the geometric progression yields:
    \[
        \rho \sum_{i=1}^{b_1 - 1} \bar\rho^{i-1} = \rho \cdot \frac{1 - \bar\rho^{b_1 - 1}}{1 - \bar\rho} = 1 - \bar\rho^{b_1 - 1},
    \]
    where we use the relation $1-\bar\rho = \rho$. Thus, the base case holds.

    Suppose the identity holds through the $(n-1)$-st buyer. Specifically, the induction hypothesis states that for the market prefix up to the index $b_{n-1}$, we have:
    \[
        \rho p^{\sigma}_{b_{n-1}} + \bar\rho \sum_{i \in B_\sigma \cap [b_{n-1}]} p^\sigma_i = \rho \sum_{i \in S_\sigma \cap [b_{n-1}]} (1 - p^\sigma_i).
    \]
    When moving to the $n$-th buyer $b_n$, the set of buyers expands by the singleton $\{b_n\}$, while the set of intermediate sellers expands by the index range $\{b_{n-1} + 1, \dots, b_n - 1\}$. Writing out the target identity at step $n$ and subtracting the induction hypothesis from it, the remaining terms that must be shown equal are exactly:
    \[
        p^{\sigma}_{b_n} - \rho p^{\sigma}_{b_{n-1}} = \rho \sum_{i=b_{n-1}+1}^{b_n-1}(1-p^{\sigma}_{i}).
    \]
    To complete the inductive step, we evaluate both sides independently using our recurrence relations. For the left-hand side, substituting $i = b_n$ into \eqref{eq:general-j} yields:
    \[
        p^{\sigma}_{b_n} = 1 - \bar\rho^{b_n - b_{n-1} - 1} (1 - \rho p^{\sigma}_{b_{n-1}}).
    \]
    Subtracting $\rho p^{\sigma}_{b_{n-1}}$ from both sides and factoring out the common term $(1 - \rho p^{\sigma}_{b_{n-1}})$ gives:
    \[
        p^{\sigma}_{b_n} - \rho p^{\sigma}_{b_{n-1}} = (1 - \rho p^{\sigma}_{b_{n-1}}) - \bar\rho^{b_n - b_{n-1} - 1} (1 - \rho p^{\sigma}_{b_{n-1}}) = (1 - \rho p^{\sigma}_{b_{n-1}}) \left(1 - \bar\rho^{b_n - b_{n-1} - 1}\right).
    \]
    For the right-hand side, applying \eqref{eq:general-j} for each intermediate seller $b_{n-1} + 1 \le i \le b_n - 1$, we can express the complementary probability as $1 - p^\sigma_i = \bar\rho^{i - b_{n-1} - 1} (1 - \rho p^{\sigma}_{b_{n-1}})$. Summing this geometric sequence leads to:
    \begin{align*}
        \rho \sum_{i=b_{n-1}+1}^{b_n-1}(1-p^{\sigma}_{i})
        &= \rho (1 - \rho p^{\sigma}_{b_{n-1}}) \sum_{m=0}^{b_n - b_{n-1} - 2} \bar\rho^m \\
        &= \rho (1 - \rho p^{\sigma}_{b_{n-1}}) \cdot \frac{1 - \bar\rho^{b_n - b_{n-1} - 1}}{1 - \bar\rho} \\
        &= (1 - \rho p^{\sigma}_{b_{n-1}}) \left(1 - \bar\rho^{b_n - b_{n-1} - 1}\right),
    \end{align*}
    where the factor $\rho$ cancels with $1-\bar\rho$. Since the two sides match identically, the induction step is complete.
\end{proof}

Now we are ready to prove Lemma~\ref{lma:alg-general}.
\begin{proof}[Proof of Lemma~\ref{lma:alg-general}]
    Let $Z_i$ indicate whether the trader holds an item immediately before processing the $i$-th agent. Then $\be[Z_i\mid\sigma]=p_i^\sigma$, and, conditional on $\sigma$, $Z_i$ is independent of the current price and its tie-breaking randomization. Therefore,
    \begin{align*}
        \be[\mathrm{ALG}]
        &=\be\left[\sum_{i\in B_{\sigma}}Z_iX_{\sigma(i)}\bo_{X_{\sigma(i)}>T}
        -\sum_{i\in S_{\sigma}}(1-Z_i)X_{\sigma(i)}\bo_{X_{\sigma(i)}<T}
        +Z_{b_n}X_{\sigma(b_n)}\bo_{X_{\sigma(b_n)}<T}\right]\\
        &=\be\left[\sum_{i\in B_{\sigma}}p_i^{\sigma}\right]\be[X_1\cdot\bo_{X_1>T}]
        -\be\left[\sum_{i\in S_{\sigma}}(1-p_i^{\sigma})-p_{b_n}^{\sigma}\right]\be[X_1\cdot\bo_{X_1<T}].
    \end{align*}
    By Lemma~\ref{lma:inventory-balance}, we have
    \[
        \sum_{i\in S_{\sigma}}(1-p_i^{\sigma})-p_{b_n}^{\sigma}
        =\frac{\bar\rho}{\rho}\sum_{i\in B_{\sigma}}p_i^{\sigma}
        =\frac{m+1}{n+1}\sum_{i\in B_{\sigma}}p_i^{\sigma}.
    \]
    Finally, substituting this identity yields
    \[
        \be[\mathrm{ALG}]=\be\left[\sum_{i\in B_{\sigma}}p^{\sigma}_{i}\right]
        \cdot \left(\be[X_1\cdot \bo_{X_1>T}]
        -\frac{m+1}{n+1}\cdot \be[X_1\cdot \bo_{X_1<T}]\right). \qedhere
    \]
\end{proof}

\section{Proof of Lemma~\ref{lma:sump-integral}}\label{app:proof-holding-core}
This appendix computes the key inventory term \(\be\left[\sum_{i\in B_{\sigma}}p_i^\sigma\right]\). Using the notation from Appendix~\ref{app:proof-alg-general}, write \(B_\sigma=\{b_1,\ldots,b_n\}\) with \(b_1<\cdots<b_n\). We first express each \(p^\sigma_{b_i}\) for a fixed arrival order, and then average these expressions over the random order.
\begin{lemma}\label{lma:pbi-general}
    For each \(i\in[n]\),
    \[
        p^{\sigma}_{b_i}=1-\rho^{i-1}\bar\rho^{b_i-i}
        -\sum_{j=1}^{i-1}\rho^{i-j-1}\bar\rho^{b_i-b_j-i+j+1}.
    \]
\end{lemma}
\begin{proof}
    The case \(i=1\) follows from \eqref{eq:general-first}. Assume the result holds for \(i-1\). Then \eqref{eq:general-j} gives
    \begin{align*}
        p^{\sigma}_{b_i}
        &=1-\bar\rho^{b_i-b_{i-1}-1}+\rho \bar\rho^{b_i-b_{i-1}-1}p^{\sigma}_{b_{i-1}}\\
        &=1-\bar\rho^{b_i-b_{i-1}-1}+\rho \bar\rho^{b_i-b_{i-1}-1}
        -\rho^{i-1}\bar\rho^{b_i-i}
        -\sum_{j=1}^{i-2}\rho^{i-j-1}\bar\rho^{b_i-b_j-i+j+1}\\
        &=1-\rho^{i-1}\bar\rho^{b_i-i}
        -\sum_{j=1}^{i-1}\rho^{i-j-1}\bar\rho^{b_i-b_j-i+j+1}.\qedhere
    \end{align*}
\end{proof}
We now average the pointwise formula. Let \(J_i:=b_i-i\) be the number of sellers arriving before the \(i\)-th buyer. Combining Lemma~\ref{lma:pbi-general} with gap exchangeability gives the following exact deficit representation.
\begin{lemma}\label{lma:sump-general}
    \[
        \sum_{i=1}^{n}\be[p^{\sigma}_{b_i}]
        =n-\sum_{i=1}^{n}\bigl(1+(n-i)\bar\rho\bigr)\rho^{i-1}\be[\bar\rho^{J_i}].
    \]
\end{lemma}
\begin{proof}
    By Lemma~\ref{lma:pbi-general},
    \begin{align*}
        \sum_{i=1}^n \be[p^{\sigma}_{b_i}]
        &=n-\sum_{i=1}^n \rho^{i-1}\be[\bar\rho^{J_i}]
        -\sum_{i=1}^{n}\sum_{j=1}^{i-1}\rho^{i-j-1}
        \be[\bar\rho^{b_i-b_j-i+j+1}].
    \end{align*}
    Set \(r=i-j\) in the double sum. For each fixed \(r\in[n-1]\), there are \(n-r\) pairs \((i,j)\) with \(i-j=r\). To evaluate their expectation, let \(Y_0,\ldots,Y_n\) be the numbers of sellers in the \(n+1\) gaps before the first buyer, between consecutive buyers, and after the last buyer. The vector \((Y_0,\ldots,Y_n)\) is exchangeable, and \(b_{j+r}-b_j-r\) is the sum of \(r\) consecutive gaps. Hence \(b_{j+r}-b_j-r\) has the same distribution as \(J_r\), so
    \[
        \sum_{i=1}^{n}\sum_{j=1}^{i-1}\rho^{i-j-1}
        \be[\bar\rho^{b_i-b_j-i+j+1}]
        =
        \sum_{r=1}^{n-1}(n-r)\bar\rho\rho^{r-1}\be[\bar\rho^{J_r}].
    \]
    Substituting this expression back and renaming \(r\) as \(i\), we obtain
    \begin{align*}
        \sum_{i=1}^{n}\be[p^{\sigma}_{b_i}]
        &=n-\sum_{i=1}^{n}\rho^{i-1}\be[\bar\rho^{J_i}]
        -\sum_{i=1}^{n-1}(n-i)\bar\rho\rho^{i-1}\be[\bar\rho^{J_i}]\\
        &=n-\sum_{i=1}^{n}\bigl(1+(n-i)\bar\rho\bigr)\rho^{i-1}\be[\bar\rho^{J_i}]. \qedhere
    \end{align*}
\end{proof}

It remains to evaluate the terms \(\be[\bar\rho^{J_i}]\) in the remaining sum. The next lemma computes them using an equivalent representation of the random arrival order by independent uniform arrival times.

\begin{lemma}\label{lma:gap-moment-integral}
    For each \(i\in[n]\),
    \[
        \be[\bar\rho^{J_i}]
        =n\binom{n-1}{i-1}\int_0^1 x^{i-1}(1-x)^{n-i}(1-\rho x)^m\dx.
    \]
\end{lemma}
\begin{proof}
    Represent the random arrival order by assigning independent uniform arrival times in \([0,1]\) to all buyers and sellers. Let \(U_{(i)}\) be the arrival time of the \(i\)-th buyer. For \(U_{(i)}\) to lie near \(x\), one buyer must arrive near \(x\), exactly \(i-1\) buyers must arrive before \(x\), and the remaining \(n-i\) buyers must arrive after \(x\). Hence, for \(x\in[0,1]\), the density of \(U_{(i)}\) at \(x\) is
    \[
        n\binom{n-1}{i-1}x^{i-1}(1-x)^{n-i}.
    \]
    Conditional on \(U_{(i)}=x\), each seller arrives before the \(i\)-th buyer independently with probability \(x\). Hence \(J_i\mid U_{(i)}=x\sim \operatorname{Bin}(m,x)\). Expanding over the possible values of \(J_i\) and applying the binomial theorem,
    \[
    \begin{aligned}
        \be[\bar\rho^{J_i}\mid U_{(i)}=x]
        &=\sum_{k=0}^{m}\bar\rho^k\bp(J_i=k\mid U_{(i)}=x)\\
        &=\sum_{k=0}^{m}\binom{m}{k}(\bar\rho x)^k(1-x)^{m-k}\\
        &=(1-x+\bar\rho x)^m\\
        &=(1-\rho x)^m.
    \end{aligned}
    \]
    Therefore, by conditioning on \(U_{(i)}\) and using its density,
    \[
        \be[\bar\rho^{J_i}]=\int_0^1 n\binom{n-1}{i-1} x^{i-1}(1-x)^{n-i}(1-\rho x)^m\dx,
    \]
    which is the claimed formula.
\end{proof}

Now we are ready to prove Lemma~\ref{lma:sump-integral}.
\begin{proof}[Proof of Lemma~\ref{lma:sump-integral}]
    Since \(B_\sigma=\{b_1,\ldots,b_n\}\), the left-hand side in the lemma is \(\sum_{i=1}^{n}\be[p^\sigma_{b_i}]\). Combining Lemmas~\ref{lma:sump-general} and \ref{lma:gap-moment-integral}, we have
    \[
        \sum_{i=1}^{n}\be[p^\sigma_{b_i}]=n-n\int_0^1(1-\rho x)^m R(x)\dx,
    \]
    where \(R(x)\) collects the remaining binomial sum:
    \begin{align*}
        R(x)
        &=\sum_{i=1}^{n}(1+(n-i)\bar\rho)\binom{n-1}{i-1}(\rho x)^{i-1}(1-x)^{n-i}\\
        &=\sum_{i=1}^{n}\binom{n-1}{i-1}(\rho x)^{i-1}(1-x)^{n-i}
        +\bar\rho(n-1)(1-x)\sum_{i=1}^{n-1}\binom{n-2}{i-1}(\rho x)^{i-1}(1-x)^{n-i-1}\\
        &=(1-\bar\rho x)^{n-1}+\bar\rho(n-1)(1-x)(1-\bar\rho x)^{n-2}\\
        &=\frac{\diff}{\diff x}\left((x-1)(1-\bar\rho x)^{n-1}\right).
    \end{align*}
    The second equality separates the two terms and uses
    \[
        (n-i)\binom{n-1}{i-1}=(n-1)\binom{n-2}{i-1}.
    \]
    The third follows from the binomial theorem. Let \(G(x)=(x-1)(1-\bar\rho x)^{n-1}\). Since \(R(x)=G'(x)\), integration by parts gives
    \begin{align*}
        \int_0^1(1-\rho x)^mR(x)\dx
        &=\int_0^1(1-\rho x)^mG'(x)\dx\\
        &=\left[(1-\rho x)^mG(x)\right]_0^1
        +m\rho\int_0^1(x-1)(1-\rho x)^{m-1}(1-\bar\rho x)^{n-1}\dx\\
        &=1-m\rho K_{n,m}.
    \end{align*}
    The boundary term is \(1\), since \(G(1)=0\) and \(G(0)=-1\), and the remaining integral is \(-K_{n,m}\).
    Substituting this into the first display gives
    \[
        \sum_{i=1}^{n}\be[p^{\sigma}_{b_i}]
        =n-n(1-m\rho K_{n,m})
        =nm\rho K_{n,m}.
    \]
    For every \(a\in(0,1)\) and \(x\in[0,1]\), \(1-ax\ge (1-x)^a\). Applying this to the two factors \(1-\rho x\) and \(1-\bar\rho x\), we get
    \begin{align*}
        K_{n,m}\ge \int_0^1(1-x)^{1+\rho(m-1)+\bar\rho(n-1)}\dx=\frac{1}{2+\rho(m-1)+\bar\rho(n-1)}=\frac{1}{2\rho(m+1)},
    \end{align*}
    where the last equality uses \(\rho=(n+1)/(n+m+2)\) and \(\bar\rho=(m+1)/(n+m+2)\). Therefore,
    \[
        \sum_{i=1}^{n}\be[p^{\sigma}_{b_i}]=nm\rho K_{n,m}\ge \frac{nm}{2(m+1)}. \qedhere
    \]
\end{proof}

\end{document}